\documentclass[11pt,leqno,a4paper]{amsart}
\usepackage{amssymb,enumerate}                  
\usepackage[usenames]{color}
\usepackage{comment}

\DeclareMathOperator{\End}{End}
\DeclareMathOperator{\Aut}{Aut}
\DeclareMathOperator{\Inf}{Inf}
\DeclareMathOperator{\Res}{Res}
\DeclareMathOperator{\Ind}{Ind}
\DeclareMathOperator{\Syl}{Syl}    

\newcommand{\Irr}{{\operatorname{Irr}}}

\newcommand{\PSL}{{\operatorname{L}}}

\newtheorem{thm}{Theorem}[section]
\newtheorem{cor}[thm]{Corollary}
\newtheorem{lem}[thm]{Lemma}
\newtheorem{prop}[thm]{Proposition}

\theoremstyle{definition}

\newtheorem{hypo}[thm]{Hypothesis}
\newtheorem{rem}[thm]{Remark}
\newtheorem{ex}[thm]{Example}

\theoremstyle{remark}

\newcommand{\IZ}{\mathbb{Z}}

\newcommand{\fA}{{\mathfrak{A}}}
\newcommand{\fS}{{\mathfrak{S}}}

\let\lra=\longrightarrow

\newcommand{\modcat}{\mathsf{mod}}    

\begin{document}

\title{Endo-trivial modules for finite groups 
\\
with Klein-four Sylow 2-subgroups}

\author{Shigeo Koshitani and Caroline Lassueur}
\address{Department of Mathematics, Graduate School of Science,
Chiba University, 1-33 Yayoi-cho, Inage-ku, Chiba, 263-8522, Japan.}
\email{koshitan@math.s.chiba-u.ac.jp}
\address{FB Mathematik, TU Kaiserslautern, Postfach 3049,
         67653 Kaisers\-lautern, Germany.}
\email{lassueur@mathematik.uni-kl.de}

\thanks{}

\begin{abstract}
We study the finitely 
generated abelian group $T(G)$ of
endo-trivial $kG$-modules where $kG$ is the group algebra 
of a finite group $G$ over a field of characteristic $p>0$. 
When the representation type of the group algebra is not wild, 
the group structure of  $T(G)$ is known for the
cases where a 
Sylow $p$-subgroup $P$ of $G$ is cyclic, 
semi-dihedral and generalized quaternion.
We investigate $T(G)$, and more accurately, 
its torsion subgroup $TT(G)$ for the case where $P$ is a Klein-four group.
More precisely, we give a necessary and 
sufficient condition
in terms of the centralizers of involutions
under which $TT(G) = f^{-1}(X(N_{G}(P)))$ 
holds, where 
 $f^{-1}(X(N_{G}(P)))$ denotes the  
abelian group consisting of the $kG$-Green correspondents 
of the one-dimensional $kN_{G}(P)$-modules.
We show that the lift to characteristic zero of any indecomposable 
module in $TT(G)$ affords an irreducible ordinary character.
Furthermore, we show that the property of a module in $f^{-1}(X(N_{G}(P)))$ 
of being endo-trivial is not intrinsic to the module itself but is decided 
at the level of the block to which it belongs.
   \end{abstract}

\maketitle

\section{Introduction}

In the representation theory of finite groups,
endo-trivial modules have seen considerable interest in the last 
fifteen years.
This lead to a complete classification for 
finite $p$-groups where $p$ is a prime, 
as well as to  the determination of the Dade groups 
for all finite $p$-groups (see, for example, 
\cite{Thevenaz2007} and the references therein).
The classification of endo-trivial $kG$-modules is 
in general achieved via the description of the structure of 
the finitely generated group $T(G)$, called \emph{the group 
of endo-trivial modules of the finite group $G$},
where $k$ is an algebraically closed
field of characteristic $p$.
Let us recall that by  definition, 
a finitely generated $kG$-module $M$ is called \emph{endo-trivial}
if there is an isomorphism
$$
\End_{k}(M)\cong k_G\oplus \text{(proj)}
$$
of $kG$-modules, 
where $k_G$ is the trivial $kG$-module and 
$\text{(proj)}$ denotes a projective $kG$-module.
The structure of the group 
$T(G)$ has been determined for some classes 
of general finite groups 
(see \cite{Carlson2012, CarlsonMazzaNakano2006, 
MazzaThevenaz2007, CarlsonHemmerMazza2010, 
CarlsonMazzaThevenaz2011,
CarlsonMazzaThevenaz2013, LassueurMalleSchulte2013}, 
for example),
but  no general solution to this problem is known.

In particular the structure of $T(G)$ is known 
when the Sylow $p$-subgroups of $G$ are cyclic 
\cite{MazzaThevenaz2007}, semi-dihedral 
and generalized quaternion 
\cite{CarlsonMazzaThevenaz2013}.  
Therefore if the representation type is not wild, 
it only remains to treat the case of dihedral 
Sylow $2$-subgroups. 
In fact in this case the torsion-free rank of $T(G)$ is known by 
\cite[Theorem 3.1 and Corollary 3.2]{CarlsonMazzaNakano2009}, 
so that the main problem remaining is to describe 
the structure of the torsion subgroup $TT(G)$ of $T(G)$.

The structure of $TT(G)$ may be of different 
nature in case the Sylow $2$-subgroups of $G$ are isomorphic to  the Klein-four group 
(that is, $kG$ has {\it domestic} representation type) 
and in case the Sylow $2$-subgroups are dihedral of order 
at least $8$ (namely, $kG$ is of {\it tame}, but not domestic, representation type),
see \cite[Theorem 4.4.4]{Benson1998}.
In this note, we treat the former case.

In \cite[Theorem A]{Carlson2012}, Carlson states that in the known cases, 
but for few exceptions, $TT(G)=X(G)$, where $X(G)$ is 
the subgroup of $T(G)$ consisting of all one-dimensional $kG$-modules,
see \cite[Section 2]{MazzaThevenaz2007}.  
We shall show that groups with the Klein-four Sylow $2$-subgroup 
are exceptions in the sense of Carlson.\par 
In fact, if $P\in{\mathrm{Syl}}_2(G)$ is 
such that $P\cong C_{2}\times C_{2}$, 
then the indecomposable torsion endo-trivial 
modules coincide with 
the trivial source endo-trivial modules, or in other words the  
$kG$-Green correspondents  of  the one-dimensional 
$kN_{G}(P)$-modules which are at the same 
time endo-trivial, that is to say 
$TT(G)=T(G)\cap f^{-1}(X(N_{G}(P)))$, 
where $f$ is the Green correspondence with respect
to $(G,P,N)$. We obtain a necessary and sufficient condition
in terms of the centralizers of involutions
under which $TT(G) = f^{-1}(X(N_{G}(P)))$ occurs.
Namely, one of our main results of this paper is the following.

\begin{thm}\label{C2xC2new}
Let $p = 2$, and let $G$  be a finite group with 
a Klein-four Sylow $2$-subgroup $P \cong C_2\times C_2$. 
Let $(K,\mathcal{O},k)$ be a splitting 
$2$-modular system for $G$ and its subgroups. 
Further set $N := N_G(P)$, and let $f$ be the Green correspondence 
with respect to $(G,P,N)$. 
Now, suppose that $1a$ is any one-dimensional $kN$-module, 
and $B$ is the $2$-block of $G$ to which $f^{-1}(1a)$, 
the $kG$-Green correspondent of $1a$, belongs. Then:
\begin{enumerate}
  \item[\rm(a)]  
The $kG$-module $f^{-1}(1a)$ uniquely lifts 
to an $\mathcal{O}G$-lattice 
affording an irreducible character in \smallskip $\Irr(B)$.
  \item[\rm(b)]
Further, $f^{-1}(1a)$ is an endo-trivial $kG$-module 
if and only if it holds that
$B$ is of principal type and that 
$\phi_u(1) = 1$ for any element 
$u\in P$ with $u \not= 1$, 
where $\phi_u$ is the unique irreducible Brauer character 
of a $2$-block $B_u$ of $C_G(u)$ such that ${B_u}^G = B$ 
(notice that such a $B_u$ is 
unique since $B$ is of principal type and $B_u$ is nilpotent). 
  \item[\rm(c)]  
In particular, if $f^{-1}(1a)$ is endo-trivial 
then all the $kG$-Green correspondents of 
one-dimensional $kN$-modules which lie in $B$ are endo-trivial.
\end{enumerate}  
\end{thm}

We emphasize that 
the endo-trivial modules $f^{-1}(1a)$ 
afford {\it irreducible} ordinary characters,
and also that 
{\it endo-triviality} in this case depends only on
the block, see Theorem \ref{C2xC2new}(a) and (c). 
From the previously known cases, see e.g. \cite{CarlsonHemmerMazza2010,CarlsonMazzaNakano2006,CarlsonMazzaNakano2009,CarlsonMazzaThevenaz2011,CarlsonMazzaThevenaz2013,MazzaThevenaz2007}, 
this does  not seem to occur in general.

\begin{rem}\label{CounterExamples}
Keep the notation as in Theorem \ref{C2xC2new}.
There are finite groups $G$, both in the  case that 
$|N_G(P):C_G(P)|= 1$ (namely $G$ is $2$-nilpotent) and in the case that
$|N_G(P):C_G(P)|= 3$, which satisfy the condition 
that there exists an
indecomposable trivial source $kG$-module $f^{-1}(1a)$ which
is {\it not} endo-trivial.
In particular, we give an example where
$X(G)\cong\{0\}$ and $TT(G)\cong\IZ/3\IZ$ while
$X(N_G(P))\cong\IZ/3\IZ\oplus\IZ/3\IZ\oplus\IZ/3\IZ$.
See 
Examples \ref{2CoreM(3)} and \ref{CounterEx}.
\end{rem}

We obtain sufficient conditions
under which $f^{-1}(1a)$ in Theorem~\ref{C2xC2new} 
is endo-trivial.

\begin{cor}\label{etSuffCond}
Keep the notation as in Theorem {\rm{\ref{C2xC2new}}}. 

Let $u$ be any element of $P$ with $u \not= 1$.
   \begin{enumerate}
   \item[\rm (a)]
If \smallskip $f^{-1}(1a)$ belongs to a $2$-block of $G$ containing 
a one-dimensional $kG$-module, then $f^{-1}(1a)$ is
endo-trivial. 
In particular, if $1a$ is in the principal block of $N$,
then $f^{-1}(1a)$ is endo-trivial.
    \item[\rm (b)]
If \smallskip $O_{2'}(C_G(u))$ is abelian, then $f^{-1}(1a)$ is 
endo-trivial.
  \end{enumerate}
\end{cor}

In terms of the structure of the group $T(G)$ of 
endo-trivial module, we have the following result, which
is also one of our main results:

\begin{thm}\label{structureT(G)}
Let $p = 2$, and let $G$ be a finite group with
the Klein-four Sylow $2$-subgroup $P \cong C_2\times C_2$.
Assume further that $|N_G(P):C_G(P)|=3$. 
Then the following \smallskip holds:
\begin{enumerate}
  \item[\rm(a)]  
The group $T(\fA_{4}) \cong \IZ/3\IZ\oplus \IZ$ embeds in $T(G)$
via inflation from $N_{G}(P)/O_{2'}(N_{G}(P))\cong \fA_{4}$ to 
$N_G(P)$ followed by the Green correspondence with respect to $(G,P,N_G(P))$.
  \item[\rm(b)] 
If $P$ is self-centralizing, that is  $C_G(P) = P$, 
then \smallskip $T(G) \cong \IZ/3\IZ\oplus \IZ$.
  \item[\rm(c)] If $u\in P$ is non-trivial and $O_{2'}(C_G(u))$ 
is abelian then $T(G)\cong f^{-1}(X(N_{G}(P)))\oplus \IZ$.
    \item[\rm(d)] 
If $G= H \times G_0$, where $H$ is a $2'$-group and 
$G_{0}$ is almost simple, that is 
$G_{0}=\PSL_{2}(q)\rtimes C_{n}$ 
with $q\equiv\pm 3 \pmod{8}$ and $C_{n}$ 
is the cyclic group of odd order $n$, then it holds
that $T(G)\cong f^{-1}(X(N_{G}(P)))\oplus \IZ$ 
and that $f^{-1}(X(N_{G}(P)))\cong
\IZ/3\IZ \oplus \IZ/n\IZ \oplus H/[H,H]$.
\end{enumerate}
\end{thm}

\begin{rem}
Keep the notation as in Theorem \ref{structureT(G)}.
  \begin{enumerate}
\item[\rm (a)]
We note that a finite group $G$ 
satisfying the assumptions of Theorem~\ref{structureT(G)}(b) 
can be written as $G/O_{2'}(G)\cong \PSL_{2}(q)$ 
with $q\equiv \pm 3 \pmod{8}$.  
The groups $\PSL_{2}(q)$ with $q \equiv \pm 3 \pmod{8}$ and 
$q>3$ are in fact the only finite simple groups with a Klein-four Sylow 
$2$-subgroup $C_2\times C_2$,
see \cite{Bender1970} and \cite{Walter1969}.
\item[\rm (b)]
Also note that the case $|N_{G}(P):C_{G}(P)|=1$ is excluded 
from the above Theorem~\ref{structureT(G)},
as it is already treated in \cite{CarlsonMazzaThevenaz2011} and 
\cite{NavarroRobinson2012}. Actually in this case $T(G)\cong X(G)\oplus\IZ$.
   \end{enumerate}
\end{rem}

This paper is divided into five sections. In \S 2 we introduce
notations we shall use and fundamental general facts, and
also preliminaries on endo-trivial modules.
\S 3 contains known results about endo-trivial modules for
finite groups with Klein-four Sylow $2$-subgroups.
\S 4 deals with the Green correspondents of one-dimensional
modules over the normalizer of Klein-four Sylow $2$-subgroups.
\S 5 contains complete and detailed proofs of the main results
Theorems \ref{C2xC2new} and \ref{structureT(G)}, and
Corollary \ref{etSuffCond}.

\section{Notations and Preliminaries}\label{sec:pre}

Throught, unless otherwise specified, we always 
let $p$ denote a prime number and $G$ a finite group.
We assume that $(\mathcal K, \mathcal O, k)$ is a
splitting $p$-modular system for all subgroups of $G$, that is, 
$\mathcal O$ is a complete discrete valuation ring of
rank one such that its quotient field $\mathcal K$ is
of characteristic zero, and its residue field
$k=\mathcal O/\mathrm{rad}(\mathcal O)$ is of
characteristic $p$, and that $\mathcal K$ and $k$ are
splitting fields for all subgroups of $G$.
Moreover, modules are always finitely generated,
and $kG{\text-}\modcat$ 
denotes the category of all finitely generated left $kG$-modules.
We say that $\mathfrak M$ is an $\mathcal OG$-lattice if
$\mathfrak M$ is a left $\mathcal OG$-module such that $\mathfrak M$
is a free $\mathcal O$-module of finite rank.
We write $\Syl_{p}(G)$ for the set of all Sylow $p$-subgroups of $G$.
If $P\in\Syl_{p}(G)$, then
we denote by $f$ the Green correspondence 
with respect to $(G, P, N_{G}(P))$, see \cite[p.276]{NagaoTsushima}. 
We write $\Delta G$ for the diagonal subgroup
$\{(g,g)\in G\times G\, | \, g\in G \}$ of $G\times G$.
For a positive integer $n$, we denote by $C_n$ a cyclic group
of order $n$, and by $\mathfrak S_n$ and $\mathfrak A_n$,
respectively, 
the symmetric and alternating groups of degree $n$.
We write $Z(G)$ for the center of $G$, for elements $g, h\in G$
we write $[g,h]:=g^{-1}h^{-1}gh$ and 
$g^h := h^{-1}gh$, and we write $G^{ab}:=G/[G,G]$ 
for the abelianization of $G$.
For two $kG$-modules $M$ and $M'$, 
$M\otimes M'$ means $M\otimes_k M'$.
We denote by $k_G$ or simply by $k$ 
the trivial $kG$-module.
We write $H \leq G$ if $H$ is a subgroup of $G$.
In such a case, for $M\in kG{\text{-}}{\sf mod}$
and $N\in kH{\text{-}}{\sf mod}$, we denote by
${\mathrm{Res}}^G_H(M)$ and ${\mathrm{Ind}}_H^G(N)$,
respectively, the restriction of $M$ to $H$ and
the induced module of $N$ to $G$.
For $N\trianglelefteq G$ and an $\mathcal O(G/N)$-lattice $M$
we write ${\mathrm{Inf}}_{G/N}^G(M)$ for the inflation
of $M$ from $G/N$ to $G$.
For $H, L \leq G$, we write $G = H \rtimes L$
if $G$ is a semi-direct product of $H$ by $L$,
note that $H \trianglelefteq G$.
By a block $B$ of $G$ (or $\mathcal OG$), we mean
a block ideal of $\mathcal OG$. Then, we write
$KB$ and $kB$, respectively, for $K\otimes_{\mathcal O}B$ and
$k\otimes_{\mathcal O}B$.
We denote by $1_B$ the block idempotent of $kB$ or
sometimes of $B$.
We denote by $\mathrm{Irr}(G)$ and $\mathrm{IBr}(G)$ the sets of
all irreducible ordinary and Brauer characters of $G$, respectively.
For a block $B$ of $G$, we also denote  by
$\mathrm{Irr}(B)$ the set of all characters in $\mathrm{Irr}(G)$
which belong to $B$. Similarly for $\mathrm{IBr}(B)$.
We write $1_G$ for the trivial ordinary or Brauer character of $G$.
For an $\mathcal OG$-lattice $\mathfrak M$, we denote
by $K\mathfrak M$ and $k\mathfrak M$, respectively,
$K\otimes_{\mathcal O}\mathfrak M$ and 
$k\otimes_{\mathcal O}\mathfrak M$.

For $Q\leq G$, we define the Brauer homomorphism
${\mathrm{Br}}^G_Q : kG \rightarrow kC_G(Q)$ by
$\sum_{g\in G}\alpha_g g 
\mapsto \sum_{g\in C_G(Q)}\alpha_g g$
where $\alpha_g\in k$.
For a block $B$ of $G$ with defect group $P$
we say that $B$ is of {\it principal type} if, for any 
$Q\leq P$, ${\mathrm{Br}}^G_Q(1_B)$ is a block idempotent of $kC_G(Q)$,
see \cite[Proposition 3.1(iii)]{HarrisLinckelmann2000}.

For an $\mathcal OG$-lattice $\hat L$, we denote by
${\hat L}^\vee$ the $\mathcal O$-dual of $\hat L$, namely,
${\hat L}^\vee = {\mathrm{Hom}}_{\mathcal O}(\hat L, \mathcal O)$.
Let $M$ be a $kG$-module.
We write $M^*$ for the $k$-dual of $M$, namely,
$M^* = \mathrm{Hom}_{kG}(M,k)$.
We say that $M$ is a {\it trivial source} module if it is a
direct sum of indecomposable $kG$-modules all of whose sources are
trivial modules, see \cite[p.218]{Thevenaz}. 
It is known that a trivial source $kG$-module $M$ 
lifts uniquely to a trivial source $\mathcal OG$-lattice, so
denote it by $\hat M$, see 
\cite[Chap.4, Theorem 8.9(iii)]{NagaoTsushima}. 
Then, by $\chi_{\hat M}$ we denote the
ordinary character of $G$ afforded by $\hat M$.
For $\chi\in{\mathrm{Irr}}(G)$ and $\phi_u\in{\mathrm{IBr}}(C_G(u))$
where $u$ is a $p$-element of $G$, we denote by 
$d^u_{\chi,\phi_u}$ the {\it generalized decomposition number}
with respect to $\chi$ and $\phi_u$, see
\cite[p.327]{NagaoTsushima}.
We let $X(G)$ denote the group of one-dimensional $kG$-modules 
endowed with the tensor product $-\otimes_{k}-$, and recall that 
$X(G)\cong (G^{ab})_{p'}$.
For an integer $n\geq 1$ and a ring $R$ we denote by 
${\mathrm{Mat}}_n(R)$ the full matrix ring of $n\times n$-matrices
over $R$.

For a non-negative integer $m$ and a positive integer $n$,
we write $n_p = p^m$ if $p^m | n$ and $p^{m+1}{\not|}n$.
For  other notation and terminology,
we refer to the books \cite{NagaoTsushima} and \cite{Thevenaz}.

\subsection{Endo-trivial $kG$-modules}\label{ssec:endo}

We start by summing up some the basic properties of endo-trivial modules.

\begin{lem}\label{lem:basicET}
Let $G$ be a finite group, let $P\in \Syl_{p}(G)$ and let $H\leq G$ 
be a subgroup.
\begin{enumerate}
 \renewcommand{\labelenumi}{\rm{(\alph{enumi})}}
  \item\label{item:basicET1}
If $M\in kG{\text{-}}\modcat$ is endo-trivial, 
then $M\cong M_{0}\oplus ({\mathrm{proj}})$ where $M_{0}$ is 
indecomposable and endo-trivial.
The relation $M\sim N\Leftrightarrow M_{0}\cong N_{0}$ 
is an equivalence relation on the class of endo-trivial modules. 
We let $T(G)$ denote the resulting set of equivalence classes. 
  \item \label{item:basicETb}
The set $T(G)$, endowed with the law 
$[M]+[N]:=[M\otimes_{k}N]$ induced by the tensor 
product $\otimes_{k}$, is an abelian group called 
the \textup{group of endo-trivial modules of $G$}. 
The zero element is the class $[k]$ and $-[M]=[M^{*}]$, 
the class of the dual module.
     \item\label{item:basicETc}
Any one-dimensional $kG$-module is endo-trivial.  Thus 
identifying $X\in X(G)$ with its class $[X]\in T(G)$, 
we have $X(G)\leq T(G)$.
   \item \label{item:basicETres}\label{item:basicETd}
If $M$ is an endo-trivial $kG$-module, 
then $\Res^{G}_{H}(M) = N\oplus {\mathrm{(proj)}}$, where $N$ is 
an endo-trivial $kH$-module. Moreover, if  $H\geq P$, 
then a $kG$-module $M$ is endo-trivial if and only if 
${\mathrm{Res}}^G_H(M)$ is endo-trivial.
   \item\label{item:basicETe}
If  $p\,|\,|H|$, then the restriction induces 
a group homomorphism 
$\Res^{G}_{H}$$ :T(G)\lra T(H):[M]\mapsto[\Res^{G}_{H}(M)]$. 
Moreover if  $H\geq N_{G}(P)$, then $\Res^{G}_{H}$ is 
injective and if $M$ is an indecomposable endo-trivial 
$kG$-module, then $[\Res^{G}_{H}(M)]=[f_{H}(M)]$ where 
$f_{H}$ denotes the Green correspondence with respect 
to  $(G,P,H)$.
   \item\label{item:basicETf}
The group $T(G)$ is finitely generated, so that we may write 
$T(G)=TT(G)\oplus TF(G)$ where $TT(G)$ is the torsion 
subgroup and $TF(G)$ is a torsion-free subgroup. 
   \item \label{item:basicETg}
If $TT(P)$ is trivial, then $TT(N_{G}(P))\!=\ker{(\Res^{N_{G}(P)}_{P})}\cong X(N_{G}(P))$.
   \item\label{item:basicETh}
 If $M$ is endo-trivial, then 
$\dim_{k}(M)\equiv \pm 1 \pmod{|G|_{p}}$ if 
$p$ is odd; and 
\linebreak
$\dim_{k}(M)\equiv \pm 1 \pmod{\frac{1}{2}|G|_{2}}$ 
if $p=2$. Moreover if $M$ is indecomposable with trivial 
source, then $\dim_{k}(M)\equiv 1 \pmod{|G|_{p}}$. 
   \item\label{item:basicETi}
Endo-trivial $kG$-modules have the Sylow $p$-subgroups 
as their vertices and lie in $p$-blocks 
of full defect.
\end{enumerate}  
\end{lem}

\begin{proof} 
For (a) and (b) see \cite[Definition 2.2.]{CarlsonMazzaNakano2009}. 
For (c) see \cite[Section 2]{MazzaThevenaz2007}. 
For (d) see \cite[Proposition 2.6]{CarlsonMazzaNakano2006}. 
Statement (e) is \cite[Lemma 2.7(1)]{MazzaThevenaz2007}.
For (f), see \cite[Corollary 2.5]{CarlsonMazzaNakano2006}. 
For (g), see \cite[Lemma 2.6]{MazzaThevenaz2007} 
and \cite[Theorem 2.3(d)(ii)]{CarlsonMazzaNakano2009}.
The first statement of  (h) is 
\cite[Lemma 2.1]{LassueurMalleSchulte2013} and 
the congruence for trivial source modules is obvious.
Finally, (i) is straightforward from (h).
\end{proof}

Constructively one gets new endo-trivial modules from old ones 
by applying the Heller operator. If $M$ is a $kG$-module, we let 
$\Omega(M)$ denotes the kernel of a projective cover $P(M)\twoheadrightarrow M$
of $M$, and 
$\Omega^{-1}(M)$ denotes the cokernel of 
$M\rightarrowtail I(M)$, where $I(M)$ is
an injective hull of $M$. 
Inductively $\Omega^{n}(M):=\Omega(\Omega^{n-1}(M))$ and 
$\Omega^{-n}(M):=\Omega^{-1}(\Omega^{-n+1}(M))$ for all 
integers $n>1$ and $\Omega^{0}(M)$ is the projective-free part of $M$.
If $M$ is endo-trivial, then so are the modules $\Omega^{n}(M)$ 
for every $n\in\IZ$. In particular the modules $\Omega^{n}(k)$ 
are endo-trivial for every $n\in\IZ$. In $T(G)$ and we have 
$[\Omega^{n}(k)]=n[\Omega(k)]$ for all $n\in\IZ$ so that 
$T(G)\geq \left<[\Omega(k)]\right>\cong \IZ$.

\begin{rem}\label{rem:TF-TT}
The torsion-free rank of $T(G)$ is non-zero 
if the $p$-rank of 
$G$ is at least $2$, and can be calculated explicitly by 
\cite[Theorem 3.1]{CarlsonMazzaNakano2009}:
it depends only on the number of conjugacy classes of 
maximal elementary abelian $p$-subgroups of rank~2.\par
The problem of determining the finite group $TT(G)$ in general is 
much harder.
According to 
Lemma~\ref{lem:basicET}(e) this is linked to the computation of 
the Green correspondence. If $P\in\Syl_{p}(G)$ and $N:=N_{G}(P)$, 
then we denote by $f^{-1}(X(N))$ the set  consisting of the $kG$-Green 
correspondents of the elements of $X(N)$. Again the tensor product 
$\otimes_{k}$ induces a group structure on $f^{-1}(X(N))$ 
such that $f^{-1}(X(N))\cong X(N)$ 
(see~\cite[Proposition 4.1(d)]{Lassueur2013}), 
but we emphasise that $f^{-1}(X(N))$ is not contained in $T(G)$ 
in general. In other words, the $kG$-Green correspondent of a 
one-dimensional $kN_{G}(P)$-module is not necessarily endo-trivial, 
although it is a trivial source module. However, 
if the $p$-rank of $G$ is at least 2 and $P$ is not 
semi-dihedral, then $f^{-1}(X(N))\cap T(G)=\ker(\Res^{G}_{P})= TT(G)$. 
This follows from Lemma~\ref{lem:basicET}(g) and the fact that 
$TT(P)\cong \{0\}$ in all such cases (see \cite{Thevenaz2007}).
\end{rem}

In order to detect whether a module in $f^{-1}(X(N))$ is endo-trivial 
we have the following character-theoretic  criterion.

\begin{thm}[{}{\cite[Theorem 2.2]{LassueurMalle2014}}]  
\label{prop:torchar}
Let $V$ be an indecomposable trivial source $kG$-module
with $k$-dimension prime to $p$. Then $V$ is endo-trivial if and only if
$\chi_{\hat V}(u)=1$ for any non-trivial $p$-element $u\in G$.
\end{thm}

\begin{lem}\label{GxH-module}
Let $G$ and $H$ be two finite groups with
$p\,|\,|H|$. Assume that $M$ is an indecomposable 
trivial source endo-trivial $kG$-module. Then,
$M\otimes k_H$ is endo-trivial as a $k[G\times H]$-module
if and only if $\dim_k(M) = 1$. 
\end{lem}

\begin{proof}
Since $M$ is endo-trivial, it holds that
$p\,{\not|}\dim_k(M)$ by Lemma \ref {lem:basicET}(h)
and that  
$M^*\otimes M = k_G\oplus L$ for a projective $kG$-module $L$.
Hence,
\begin{align*}
(M\otimes k_H)^*\otimes(M\otimes k_H)
&= 
M^*\otimes {k_H}^*\otimes M \otimes k_H
\\
&= 
(M^*\otimes M)\otimes (k_H\otimes k_H) 
=  
(k_G\oplus L)\otimes k_H 
\\ 
&=  (k_G\otimes k_H)\oplus(L\otimes k_H).
\end{align*}
Since $p| |H|$, $L\otimes k_H$ is not projective 
as $k[G\times H]$-module.

Now, let $G_0$ be a subgroup of $G$. Then,
\begin{align*}
\Ind_{G_{0}}^{G}(k_{G_0})\otimes k_H 
&=
(kG\otimes_{kG_0}k_{G_0})\otimes k_H
\\
&\cong
k[G\times H]\otimes_{k[G_0\times H]}(k_{G_0}\otimes k_H)
= 
\Ind_{G_0\times H}^{G\times H}(k_{G_0\times H})
\end{align*}
This shows that $M \otimes k_H$ is an indecomposable trivial 
source $k[G\times H]$-module,
see \cite[Proposition 1.1]{Kuelshammer1993}.
So let $\chi$ be an ordinary character of $G\times H$ 
afforded by $M\otimes k_H$.
Take any $p$-elements
$u\in G$ and $v\in H$ with $u\not= 1$ and $v\not= 1$.
Then, we know by Theorem~\ref{prop:torchar} that
\begin{align*}
\chi(u,v) &= \chi_{\hat M}(u){\cdot}1_H(v) = 1{\cdot}1= 1,
\\
\chi(u,1) &= \chi_{\hat M}(u){\cdot}1_H(1) = 1{\cdot}1 = 1  
\ \ \text{ and } 
\\
\chi(1,v) &= \chi_{\hat M}(1){\cdot}1_H(v) = \chi_{\hat M}(1) = \dim_k(M).
\end{align*}
Therefore, again by making use of Theorem~\ref{prop:torchar},
the assertion follows.
\end{proof}

\begin{lem}\label{princ_type}
Let $B$ be a block of $G$ with a maximal $B$-Brauer pair
$(P,e)$, and let $N_G(P,e)$ be the stabilizer of $(P,e)$,
see \cite[p.346]{Thevenaz}.
\begin{enumerate}
  \item[\rm(a)]
If further we have $N_G(P,e)=N_G(P)$, then $e$ is the unique block of
$C_G(P)$ with $e^G=B$ (block induction).
  \item[\rm(b)]
In particular, if the Brauer correspondent of $B$ in $N_G(P)$
has a one-dimensional $kN_G(P)$-module,
then $e$ is the unique block of
$C_G(P)$ with $e^G=B$.
\end{enumerate}
\end{lem}

\begin{proof}
Set $N:=N_G(P)$ an $H:=N_G(P,e)$.

(a) We use the same notation $e$ to mean a block and its
block idempotent. In general, $e$ is a block idempotent even
of $kH$, see \cite[Exercise (40.2)(b)]{Thevenaz}, and hence of $kN$ by the assumption. Namely, $e$ is the block 
idempotent of the Brauer correspondent of $B$ in $N$.
Hence, Brauer's 1st Main Theorem implies
$e = {\mathrm{Br}}^G_P(1_B)$. Now, let $e'$ be any block
idempotent of $kC_G(P)$ with $(e')^G=B$ (block induction).
Then, by \cite[Chap.5, Theorem 3.5(ii)]{NagaoTsushima},
$e' = {\mathrm{Br}}^G_P(1_B){\cdot}e' = ee'$, which yields $e'=e$.

(b) By the result of Fong-Reynolds
\cite[Chap.5, Theorem 5.10(ii)]{NagaoTsushima}, 
$kNb\cong {\mathrm{Mat}}_{|N:H|}(kHe)$ 
as $k$-algebras where
$b$ is the block idempotent of the Brauer correspondent of 
$B$ in $N$. 
Since $b$ has a one-dimensional $kN$-module, we have $|N:H|=1$,  
so the assertion follows by (a).
\end{proof}

\subsection{Puig equivalences}\label{ssec:puigequiv}

\begin{lem}\label{Mathcal_O_K}
Let $G$ and $H$ be finite groups, and let $\mathfrak M$ and
$\mathfrak N$, respectively, be an 
$(\mathcal OG,\mathcal OH)$- and
$(\mathcal OH,\mathcal OG)$-bilattices. Then,
for any $R\in\{ K, k \}$, we have
$$
R\otimes_{\mathcal O}(\mathfrak M\otimes_{\mathcal OH}\mathfrak N)
\cong (R\otimes_{\mathcal O}\mathfrak M)
   \otimes_{KH} 
      (R\otimes_{\mathcal O}\mathfrak N) \ 
\text{as }(RG,RG)\text{{\rm{-}}bimodules}.
$$
\end{lem}

\begin{proof}
This follows  by direct calculations,
see e.g. \cite[Exercise (9.7)]{Thevenaz} 
and \cite[I Lemma 14.5(3)]{Landrock1983}.
\end{proof}

\begin{lem}\label{Puig_tsm}
Let $G$ and $G'$ be finite groups, and let $B$ and $B'$, respectively,
be block algebras of $\mathcal OG$ and $\mathcal OG'$ with the same
defect group $P$ (and hence $P \leqslant G\cap G')$.
Assume that $B$ and $B'$ are Puig equivalent. Namely,
there exist an $(\mathcal OG, \mathcal OG')$-lattice $\mathfrak M$ 
and an $(\mathcal OG', \mathcal OG)$-lattice $\mathfrak N$
such that the pair $(\mathfrak M, \mathfrak N)$ 
induces a Morita equivalence between $B{\text{-}}{\sf mod}$ and
$B'{\text{-}}{\sf mod}$,
$\mathfrak N = {\mathfrak M}^\vee$, and $\mathfrak M$
is an indecomposable trivial source
lattice with vertex $\Delta P$ as an $\mathcal O[G\times G']$-lattice.
Then, if $V$ is an indecomposable trivial source 
$kG$-module in $B$ with vertex $Q$ such that $Q \leqslant P$,
it holds that $k\mathfrak N\otimes_{kG}V$ is an indecomposable
trivial source $kG'$-module in $B'$ with vertex $Q$ 
(if we replace $Q$ by ${g'}^{-1}Qg'$ for some element $g'\in G'$).
\end{lem}

\begin{proof}
This follows from exactly the same argument given in
the proof of \cite[Lemma A.3 (i), (iii)]{KoshitaniMuellerNoeske2011},
though $G'$ is not necessarily a subgroup of $G$.
See also \cite[Remark 7.5]{Puig1999} and \cite[Theorem 4.1]{Linckelmann2001}.
\end{proof}

\section{Endo-trivial modules for groups with Klein-four Sylow subgroups}

\begin{hypo}
Henceforth we assume $p = 2$, 
so that $k$ is an algebraically closed field of characteristic $2$. 
We assume further that $G$ has a Klein-four
Sylow $2$-subgroup $P \cong C_2\times C_2$.
Set $N:=N_{G}(P)$ and $\bar{N}:=N/O_{2'}(N)$. 
Recall that $\Aut(P) \cong \fS_{3}$,
(see \cite[Chap.7, Theorem 7.1]{Gorenstein1968}), 
so that  one of the following holds: 
\begin{enumerate}
\renewcommand{\labelenumi}{\rm{(\roman{enumi})}}
  \item  $|N_{G}(P):C_{G}(P)|=1$ and there are three conjugacy classes 
of involutions. By Burnside transfer theorem and its converse,
this happens if and only if  $G$ is $2$-nilpotent, that is 
$\bar{N}\cong C_{2}\times C_{2}$\,; or
  \item  $|N_{G}(P):C_{G}(P)|=3$ and all involutions are conjugate in $G$. 
In this case $\bar{N}\cong \fA_{4}$. 
  \end{enumerate} 
\end{hypo}

Next we sum up what is known on the structure of the group 
$T(G)$ in previously published articles.

\begin{lem}\label{lem:TGKlein}
Let $G$ be a finite group having a Sylow $2$-subgroup 
$P\cong C_{2}\times C_{2}$. 
\begin{enumerate}
\renewcommand{\labelenumi}{\rm{(\alph{enumi})}}
  \item
$T(C_{2}\times C_{2})=\left<[\Omega(k)]\right>\cong \IZ$;
  \item
$T(G)=TT(G)\oplus \IZ$, 
where $\IZ\cong \left<[\Omega(k)]\right>$ 
and 
$$TT(G)=\ker(\Res^{G}_{P})=f^{-1}(X(N_{G}(P)))\cap T(G)$$ consists of 
the classes of the indecomposable 
endo-trivial modules with a trivial source, 
or alternatively 
the classes of 
the indecomposable endo-trivial $kG$-modules whose 
$kN_{G}(P)$-Green correspondent is one-dimensional;
 \item
$T(\fA_{5})\cong T(\fA_{4})=X(\fA_{4})
\oplus \left<[\Omega(k)]\right>\cong \IZ/3\IZ\oplus \IZ$, 
where $X(\fA_{4})\cong  \IZ/3\IZ$ consists of the three 
one-dimensional $k\fA_{4}$-modules, which we denote by 
$k_{\fA_{4}}$, $k_{\omega}$, $k_{\overline{\omega}}$;
 \item If $G=O_{2'}(G)\rtimes P$, then $T(G)\cong X(G)\oplus \IZ$.
\end{enumerate}  
 \end{lem}

\begin{proof}
Part (a) was proved by Dade \cite[Theorem 9.13]{Dade1978}. 
For (b), it is  straightforward from 
\cite[Theorem 3.1]{CarlsonMazzaNakano2009} 
that $TF(G)\cong \IZ$ generated by $[\Omega(k)]$. Now since $TT(P)=\{[k]\}$ 
by (a), clearly we must have $TT(G)=\ker(\Res^{G}_{P})$. The rest follows 
directly from Lemma \ref{lem:basicET}(d), (e) and (g).
For (c) it follows from (b) and Lemma~\ref{lem:basicET}(h) that 
$T(\fA_{4})=X(\fA_{4})\oplus \left<[\Omega(k)]\right>$ and 
$X(\fA_{4})\cong (\fA_{4}/[\fA_{4}, \fA_{4}])_{2'}\cong\IZ/3\mathbb Z$. 
In addition $T(\fA_{5})\cong T(\fA_{4})$ since $\fA_{4}$ is 
strongly $2$-embedded in $\fA_{5}$,
see \cite[Theorem 4.2]{CarlsonMazzaNakano2009}.
Part (d) was conjectured in \cite{CarlsonMazzaThevenaz2011} 
and proven in \cite{NavarroRobinson2012}. 
\end{proof}

Thanks to Lemma~\ref{lem:TGKlein}(b), 
it remains to understand the group $TT(G)$ in general.
In the next section we describe a necessary and 
sufficient condition under which we have 
$$TT(G)=f^{-1}(X(N_{G}(P)))\,.$$
However, we point out that even in the $2$-nilpotent case  
(see Lemma~\ref{lem:TGKlein}(d))
it may happen that the Green correspondence $f^{-1}$ 
does not preserve endo-trivial modules, that is  
$TT(G)\lneqq f^{-1}(X(N_{G}(P)))$.
The following are typical examples of such situations.

\begin{ex}\label{2CoreM(3)}
Let $P$ be a Klein four-group with generators $u$ and $v$, that is 
$\langle u\rangle\times\langle v\rangle \cong C_2\times C_2$. 
Let $\ell$ be an odd prime, and $H$ denote an 
extra-special group of order $\ell^{3}$ and exponent $\ell$, 
given by the presentation 
$$H = \langle  a,b,c \, | \, a^{\ell}=b^{\ell}=c^{\ell}=1, 
[a,c]=1=[b,c],[a,b]=c\rangle\,.$$
Then build $G := H\rtimes P$ to be the semi-direct 
product where $P$ acts on $H$ as follows: $v$ 
acts trivially, that is $[H,v] = 1$,   $a^{u}= a^{-1}$, $b^u = b^{-1}$, 
$c^u = c$.  
We have $H = O_{2'}(G)$ and 
$N_{G}(P)=C_{G}(P)=Z\times P$ where 
$Z:=<c>\cong C_{\ell}$ is the center of $H$.\par
By Lemma~\ref{lem:TGKlein}(d), 
$T(G)\cong X(G)\oplus\IZ$ and, 
by Lemma~\ref{lem:basicET}(e), it embeds as a subgroup of  
$T(N_{G}(P))\cong X(N_{G}(P))\oplus \IZ$ via the restriction map
$\Res^{G}_{N_{G}(P)}:T(G)\lra T(N_{G}(P))$. Now 
$$X(G)=(G^{ab})_{2'}=  \left[   \left ((H)^{ab} \right)_{P}
\times P^{ab}  \right]_{2'}= (H/Z)_{P}=<aZ,bZ>_{P}=1$$
whereas  
$X(N_{G}(P))=(N_{G}(P)^{ab})_{2'}=Z_{P}=Z\cong C_{\ell}\,.$ 
(The subscript $P$ denotes taking the co-invariants 
with respect to the action of $P$, that is for a group 
$U$ which is a $P$-set then  
$U_{P}:=U/\!<\!u^{p}u^{-1}\,|\,u\in U,p\in P\!>$). 
In consequence $TT(G)\lneqq TT(N_{G}(P))
\cong f^{-1}(X(N_{G}(P)))$.
\end{ex}

\begin{ex}\label{CounterEx_2Nil}
Let us reconsider Example \ref{2CoreM(3)} and restrict ourselves
to the situation $\ell := 3$, and let $P$, $u$, $v$ and $G$ be the same
as in Example \ref{2CoreM(3)}. 
Set $H := 3_+^{1+2}$, namely, $H$ is an extra-special
group of order $3^3$ with exponent $3$, see \cite[p.xx]{Atlas}.
Then, $G$ is $2$-nilpotent and $H = O_{2'}(G)$. Set 
$Z := Z(H)$ and $N := N_G(P)$, so that $Z \cong C_3$ and
$$
G = (H\rtimes\langle u\rangle)\times\langle v\rangle
  = H \rtimes P = C_G(v), \qquad\qquad
N = Z \times P.
$$ 
Set $Z := \langle z \rangle$.
There exist exactly two characters in ${\mathrm{Irr}}(H)$ with
degree $3$. Take one of them and denote it by 
$\theta\in{\mathrm{Irr}}(H)$, namely $\theta(1) = 3$.
Since $z^u = z$ and $z^v = z$, we have $T_G(\theta) = G$,
where $T_G(\theta)$ is the inertial group of $\theta$ in $G$. 
Since $G$ is $2$-nilpotent, it follows from
Morita's theorem \cite[Theorem 2]{Morita1951} (see
\cite[Lemma 2]{Koshitani1982}) and
\cite{FongGaschuetz1961} that
there is a block $B$ of $\mathcal OG$ with
$$
   kB \ \cong \ {\mathrm{Mat}}_{\theta(1)}(kP) \ = \
                {\mathrm{Mat}}_3(kP),
   \qquad \text{as }k\text{-algebras}.
$$
Namely, $\theta$ is covered by $B$, and $B$ is a non-principal
block with defect group $P$. Let $b$ be a block of $\mathcal ON$
with $b^G = B$.
Since $N = Z\times P$ and this is abelian, $kb \cong kP$
as $k$-algebras. Hence, the unique simple $kN$-module $W$ 
in $b$ is of dimension one.
Set $V := f^{-1}(W)$.
Note that $\mathcal OP$ is a source algebra of $B$ at least 
by making use of
\cite[Theorem 1.1]{CravenEatonKessarLinckelmann}.
Hence, by the proof of Case 1 of Proposition~\ref{IrrChar}(a),
$V$ is the unique simple $kG$-module in $B$.
Set $\chi := \chi_{\hat V}$.
Then, by Proposition~\ref{IrrChar}(a), $\chi\in{\mathrm{Irr}}(B)$.
It follows from  the character table of $G$  that
$\chi(1) = 3$, $\chi(u) = 1$, $\chi(v) = 3$ and $\chi(uv) = 1$.
Then, it follows from Theorem \ref{prop:torchar} that
$V$ is not endo-trivial 
(or since $\chi(1) = \dim_k(V) \not\equiv 1$ (mod $4$)
it follows from Lemma \ref{lem:basicET}(h) that $V$ is not endo-trivial).

We also obtain the same conclusion using Theorem \ref{C2xC2new}. Namely,
since $G = C_G(v)$, we get $\phi_v = \phi$ where $\phi$ is 
the unique irreducible Brauer character in $B$ and $\phi_v$ is the
same as in Lemma \ref{IrrChar}(b). Since 
$kB \cong {\mathrm{Mat}}_3(kP)$, $\phi_v(1) = \phi(1) = 3$, which 
means $\phi_v(1) \not= 1$. Therefore, Theorem \ref{C2xC2new}
yields that $V$ is not endo-trivial,
 
As a matter of fact, $kG$ has the following block decomposition:
$$
kG \cong kP \oplus 4\times{\mathrm{Mat}}_2(k\langle v\rangle)
            \oplus 2\times{\mathrm{Mat}}_3(kP)
\qquad \text{as }k\text{-algebras},
$$
by making use of \cite[Theorem 2]{Morita1951}
(see \cite[Lemma2]{Koshitani1982}) and \cite{FongGaschuetz1961}.
\end{ex}

\begin{ex}\label{CounterEx}
We now give an example of a finite group $G$,  
which is not $2$-solvable, has a Sylow $2$-subgroup 
$P\cong C_{2}\times C_{2}$, and such that there exists 
a non-trivial one-dimensional $kN_{G}(P)$-module $1a$ 
such that its $kG$-Green correspondent $f^{-1}(1a)$ is 
{\it not} endo-trivial.\par
We rely here on  computations using the algebra software {\sf MAGMA} \cite{MAGMA}.
We know from the {\sf ATLAS} \cite[p.207]{Atlas}
that the sporadic simple group $(Fi_{24})'$, 
which is the commutator subgroup of 
the  Fischer  group 
$Fi_{24}$, has a maximal subgroup $N(3B)$ with
$N(3B) \cong 3_+^{1+10}:(U_5(2):2)$. Then, we compute that
$N(3B)$ has a maximal subgroup $G$ such that
$|N(3B):G| = 2^8{\cdot}3^4 = 20\,736$, 
$|G| = 2^2{\cdot}3^{12}{\cdot}5{\cdot}11\, = 116 \, 917\, 020$, 
and has $G$ a Sylow $2$-subgroup $P$ with $P \cong C_2\times C_2$.
Furthermore, $|N_G(P)| = 2^2{\cdot}3^6$ and $|C_G(P)| = 2^2{\cdot}3^5$.\par   
Using {\sf MAGMA} to compute the character tables of 
$G$ and $N:=N_{G}(P)$, we gather the following information. 
On the one hand,  we know that $G$ has  exactly 
three irreducible characters 
$\chi_{1},\chi_{2},\chi_{3}\in {\mathrm{Irr}}(G)$ such that 
$\chi_{1}(u)=\chi_{2}(u) =\chi_{3}(u)= 1$ for any $u\in P-\{1\}$. 
Therefore we conclude from 
Theorem~\ref{C2xC2new}(a), Corollary~\ref{etSuffCond}(a), 
and Theorem~\ref{prop:torchar} that $\chi_{1},\chi_{2},\chi_{3}$ 
are the characters afforded by three non-isomorphic trivial 
source endo-trivial modules $V_{1}:=k_{G}$, $V_{2}$, $V_{3}$ 
in the principal block  $B:=B_0(G)$ of $G$, and moreover that 
these are the only trivial source endo-trivial modules of $kG$, 
up to isomorphism. 
On the other hand, there are exactly twenty-seven non-isomorphic
simple $kN$-modules, say $1_1 := k_N$, $1_2$, $\cdots$, $1_{27}$ 
and $X(N)\cong \IZ /3\IZ\oplus\IZ /3\IZ\oplus\IZ /3\IZ$.
So, we have twenty-seven non-isomorphic indecomposable
trivial source $kG$-modules $f^{-1}(1_i)$ with vertex $P$
for $i = 1, \cdots, 27$.

More precisely, the three non-isomorphic
endo-trivial $kB$-modules $V_{1}$, $V_{2}$ and $V_{3}$ are all simple, 
and their $k$-dimensions
are $1$, $5$, $5$, so that the principal block $B$  is Puig equivalent
to $\mathcal O\mathfrak A_4$.
Moreover, by Brauer's 1st Main Theorem and the
character table of $G$, we know that there are exactly eight other blocks, 
say $B_2, \cdots, B_9$, of $G$
with full defect and all of them have $kG$-modules
of the form $f^{-1}(1a)$, that is the Green correspondents
of one-dimensional $kN$-modules, but none of these are endo-trivial modules.\par
As a result, this example shows that
$X(G) \cong \{0\}$, $TT(G) \cong \IZ /3\IZ$
and $X(N_G(P))\cong \IZ /3\IZ\oplus\IZ /3\IZ\oplus\IZ /3\IZ$.\par
A similar example can be obtained from a maximal subgroup $H$ of 
$G$ of order $481\,140=2^{2}\cdot3^{7}\cdot5\cdot 11$ in which 
case we obtain $X(G)\cong \IZ/3\IZ$, $TT(G) \cong \IZ /3\IZ\oplus\IZ /3\IZ$ 
and $X(N_H(P))\cong \IZ /3\IZ\oplus\IZ /3\IZ\oplus\IZ /3\IZ$.
\end{ex}

\section{Green correspondents of one dimensional module over $N_G(P)$}

\begin{lem}\label{GreenCorrA5}
Let $G := \mathfrak A_5$ and let 
$B_0 := B_0(\mathcal O\mathfrak A_5)$
be the principal block algebra of $\mathcal O\mathfrak A_5$.
Then,  the indecomposable trivial source $kG$-modules in $B_0$
with vertex $P$ are $k_G$ and two uniserial modules
$$
    \boxed{\begin{matrix} 2a  \\
                          k_G \\
                          2b
           \end{matrix}
          }\,,
\qquad
    \boxed{\begin{matrix} 2b  \\
                          k_G \\
                          2a
           \end{matrix}
          }\,,
$$
where $2a, 2b$ are nonisomorphic simple $kG$-modules in $B_0$
of $k$-dimension two. These three indecomposable modules 
lift uniquely to trivial source
$\mathcal OG$-lattices which afford ordinary characters 
$1_{G}$, $\chi_5$ and $\chi_5$, respectively, 
where  $\chi_5\in\Irr(B_0)$
is of degree five.
\end{lem}

\begin{proof}
This follows from direct calculations of the Green correspondents 
$f^{-1}(1a)$ of three simple $kN_G(P)$-modules $1a$
of $k$-dimension one (recall that $N_G(P) = \mathfrak A_4$).
Then, the assertions follow by 
\cite[Chap.4, Problem 10]{NagaoTsushima}.
\end{proof}

\begin{lem}\label{GenDecNumber}
Let $B$ be a block of $\mathcal OG$ with a defect group 
$P \cong C_2\times C_2$ such that $B$ is one of $\mathcal OP$,
$\mathcal O\mathfrak A_4$ or $B_0(\mathcal O\mathfrak A_5)$,
and hence $G$ is $P$, $\mathfrak A_4$ or $\mathfrak A_5$, respectively. 
Assume $\chi\in{\mathrm{Irr}}(B)$ such that
\begin{equation*}
\chi \ \in \ 
\begin{cases}
\{1_P\}, & \text{if } B = \mathcal OP, \\
\{ 1_{\mathfrak A_4}, \chi_{\omega}, \chi_{\bar\omega} \},
     & \text{if } B = \mathcal O\mathfrak A_4, \\
\{ 1_{\mathfrak A_5}, \chi_5 \},
     & \text{if } B = B_0(\mathcal O\mathfrak A_5),
\end{cases}
\end{equation*}
where $\chi_{\omega}$ and $ \chi_{\bar\omega}$ are two different
non-trivial ordinary characters of $\mathfrak A_4$ of degree one, and
$\chi_5$ is the same as in Lemma {\rm\ref{GreenCorrA5}}.
Then, for any $u\in P$ with $u \not= 1$, it holds
$$
d^{u}_{\chi,\phi_u} \ = \ 1
$$
where $\phi_u$ is an irreducible Brauer character of $C_G(u)$ such
that ${B_u}^G = B$ for a block $B_u$ of $C_G(u)$ which satisfies
${\mathrm{IBr}}(B_u) = \{\phi_u\}$ (note that $B_u$ is nilpotent).
\end{lem}

\begin{proof}
For the case $B = \mathcal OP$, the claim is obvious.

Next assume that $B = \mathcal O\mathfrak A_4$. Then the generalized
$2$-decomposition matrix is the following:
\begin{center}
{\rm
\begin{tabular}{l|cccr}
  & $\phi_1$ & $\phi_{\omega}$ & $\phi_{\bar\omega}$ & $\phi_u$ \\
\hline
$\chi_1$            & $1$ & .   & .   & $1$  \\
$\chi_{\omega}$     & .   & $1$ & .   & $1$  \\
$\chi_{\bar\omega}$ & .   & .   & $1$ & $1$  \\ 
$\chi_{3}$          & $1$ & $1$ & $1$ & $-1$ \\ 
\end{tabular} 
}
\end{center}
where ${\mathrm{IBr}}(B) = \{ \phi_1, \phi_{\omega}, \phi_{\bar\omega} \}$,
${\mathrm{Irr}}(B) = 
\{ \chi_1, \chi_{\omega}, \chi_{\bar\omega}, \chi_3 \}$
and ${\mathrm{IBr}}(B_u) = \{\phi_u\}$.
Hence the assertion holds.

Suppose that $B = B_0(\mathcal O\mathfrak A_5)$. Then, the generalized
$2$-decomposition matrix is
\begin{center}
{\rm
\begin{tabular}{l|cccr}
  & $\phi_1$ & $\phi_{2a}$ & $\phi_{2b}$ & $\phi_u$ \\
\hline
$\chi_1$            & $1$ & .   & .   & $1$  \\
$\chi_{3a}$         & $1$ & $1$ & .   & $-1$  \\
$\chi_{3b}$         & $1$ & .   & $1$ & $-1$  \\ 
$\chi_{5}$          & $1$ & $1$ & $1$ & $ 1$ \\ 
\end{tabular} 
}
\end{center}
where ${\mathrm{IBr}}(B) = \{ \phi_1, \phi_{2a}, \phi_{2b} \}$,
${\mathrm{Irr}}(B) = 
\{ \chi_1, \chi_{3a}, \chi_{3b}, \chi_5 \}$
and ${\mathrm{IBr}}(B_u) = \{\phi_u\}$.
Thus, this implies the assertion.
\end{proof}

\begin{prop}\label{IrrChar}
Let $G$ be a finite group with a Sylow $2$-subgroup $P$ such that
$P \cong C_2 \times C_2$, and set $N := N_G(P)$.
Let $V := f^{-1}(1a)$ be the $kG$-Green correspondent of a 
one-dimensional $kN$-module $1a$,  and set $\chi := \chi_{\hat V}$.
Then, the following holds:
\begin{enumerate}
\renewcommand{\labelenumi}{\rm{(\alph{enumi})}}
    \item
$\chi$ is irreducible, namely, $\chi\in{\mathrm{Irr}}(G)$.
   \item
Let $b$ be the block of $\mathcal ON$ to which $1a$  belongs,
and set $B := b^G$ (block induction).
Then, for any $u\in P$ with $u \not= 1$, 
and for any block $B_u$ of $\mathcal OC_G(u)$
with ${B_u}^G = B$, we have that $B_u$ is nilpotent and that
$$
          d^{u}_{\chi, \phi_u} \ = \ 1
$$
where $\phi_u$ is the unique irreducible Brauer character in $B_u$,
that is ${\mathrm{IBr}}(B_u) = \{ \phi_u \}$.
\end{enumerate}
\end{prop}

\begin{proof}
(a)
Let $b$ and $B$ be as in (b) above.
Clearly, $P$ is a defect group of $b$ and hence of $B$ as well.
It is known that $V$ belongs to $B$, see
\cite[Chap.5, Corollary 3.11]{NagaoTsushima}.
Recall that all indecomposable trivial source $kG$-modules in $B$
with vertex $P$ are the Green correspondents of simple $kN$-modules 
in $b$, see \cite[Chap.4, Problem 10]{NagaoTsushima} and
\cite[II Lemma 10.3]{Landrock1983}.

Now, it follows from \cite[Theorem 1.1]{CravenEatonKessarLinckelmann}
that $B$ is Puig equivalent to one of the three block algebras
$B_0(\mathcal O\mathfrak A_5)$, 
$\mathcal O\mathfrak A_4$ or $\mathcal OP$,
where $B_0(\mathcal   O\mathfrak A_5)$ is the principal block algebra of
$\mathcal O\mathfrak A_5$. 
Let $G'$ be one of the three groups $\mathfrak A_5$, $\mathfrak A_4$
or $P$, and let $B'$ be one of the three block algebras 
$B_0(\mathcal O\mathfrak A_5)$, $\mathcal O\mathfrak A_4$ or
$\mathcal OP$, depending on the situation.
Since $1a$ is of dimension one, $1a$ is an
indecomposable trivial source $kN$-module with vertex $P$.

Since $B$ and $B'$ are Puig equivalent,
there exist an $(\mathcal OG, \mathcal OG')$-lattice $\mathfrak M$ 
and an $(\mathcal OG', \mathcal OG)$-lattice $\mathfrak N$ 
such that $\mathfrak N = {\mathfrak M}^\vee$ ($\mathcal O$-dual of 
$\mathfrak M$), 
$\mathfrak M$ is an indecomposable trivial source
lattice with vertex $\Delta P$ as an $\mathcal O[G\times G']$-lattice
and the pair $(\mathfrak M, \mathfrak N)$ induces 
a Morita equivalence between $B$ and $B'$.

Note that the $(RB, RB')$-bimodule $R\mathfrak M$ induces
a Morita equivalence between $RB$ and $RB'$
for $R\in\{ K,k\}$,
see \cite[Exercise (9.7)]{Thevenaz}. 

{\bf Case 1:} $B' = \mathcal OP$. Hence $G' = P$ in this case.
Recall that $\mathfrak M$ takes a trivial source simple module to
a trivial source simple module with the same vertex by Lemma \ref{Puig_tsm}.
Let $S$ be the unique simple $kG$-module in $B$.
Then, clearly, $k\mathfrak N\otimes_{kP}k_P = S$.
So, $S$ is a trivial source simple $kG$-module in $B$ with
vertex $P$ as remarked above. 
Thus, $S = f^{-1}(1a)$ since $f^{-1}(1a)$
is the unique indecomposable trivial source $kG$-module in $B$ with
vertex $P$ by Lemma \ref{Puig_tsm}.
Namely, we may assume $V = S$.
Now, just as in Lemma \ref{Mathcal_O_K},
$$
k\otimes_{\mathcal O}(\mathfrak N\otimes_{\mathcal OG}\hat V)
=
(k\otimes_{\mathcal O}\mathfrak N)\otimes_{kG}
(k\otimes_{\mathcal O}\hat V)
=
k\mathfrak N\otimes_{kG}V = k_P = k\otimes_{\mathcal O}\mathcal O_P.
$$
Because of the uniqueness of the lifting from $k$ to $\mathcal O$
for trivial source modules
(see \cite[Chap.4, Theorem 8.9(iii)]{NagaoTsushima}), 
$\mathfrak N\otimes_{\mathcal OG}\hat V = \mathcal O_P$.
This implies that
$$
K\otimes(\mathfrak N\otimes{\mathcal OG}\hat V)
=
K\otimes_{\mathcal O}\mathcal O_P = K_P =: \chi_{1_P}.
$$
Since 
$
K\otimes_{\mathcal O}(\mathfrak M\otimes_{\mathcal OP}\mathfrak N)
\cong
(K\otimes_{\mathcal O}\mathfrak M)\otimes_{KP}
     (K\otimes_{\mathcal O}\mathfrak N)
$
as $(KG,KG)$-bimodules by Lemma \ref{Mathcal_O_K},
these yield 
\begin{align*}
\chi_{\hat V} 
&:= 
K\otimes_{\mathcal O}\hat V
=
K\otimes_{\mathcal O}((\mathfrak M\otimes_{\mathcal OP}\mathfrak N)
          \otimes_{\mathcal OG}\hat V )
\\
&=
(K\otimes_{\mathcal O}\mathfrak M)\otimes_{KP}
 (K\otimes_{\mathcal O}(\mathfrak N\otimes_{\mathcal OG}\hat V))
= K\mathfrak M\otimes_{KP}K_P.
\end{align*}
Therefore, the ordinary character $\chi_{\hat V}$ of $G$ afforded by $V$ is
irreducible, see \cite[Exercise (9.7)]{Thevenaz}.

{\bf Case 2:} $B' = \mathcal O\mathfrak A_4$. 
Hence $G' = \mathfrak A_4$ in this case.
First, since $P \trianglelefteq \mathfrak A_4$,
all indecomposable trivial source $k\mathfrak A_4$-modules
with vertex $P$ are the three simple $k\mathfrak A_4$-modules
of $k$-dimension one, which we denote by $k_{\mathfrak A_4}$, $k_{\omega}$
and $k_{\bar\omega}$.
Thus, as in Case 1, all the indecomposable trivial source
$kG$-modules in $B$ with vertex $P$ are exactly
$k\mathfrak M\otimes_{k\mathfrak A_4}k_{\mathfrak A_4}$,
$k\mathfrak M\otimes_{k\mathfrak A_4}k_{\omega}$, and
$k\mathfrak M\otimes_{k\mathfrak A_4}k_{\bar\omega}$.
In fact, these are all the simple $kG$-modules in $B$ as well. 

Now, since $|N_G(P)/C_G(P)| = 3$, $b$ is also Puig equivalent to
$\mathcal O\mathfrak A_4$ by looking at the $2$-decomposition 
matrices of $B_0(\mathcal O\mathfrak A_5)$, 
$\mathcal O\mathfrak A_4$ and $\mathcal OP$. Namely, $b$ has
exactly three simple $kN$-modules, which we denote by
$W =: W_0$, $W_1$ and $W_2$. As remarked above, all indecomposable
trivial source $kG$-modules in $B$ with vertex $P$ are
$f^{-1}(W_i)$ for $i = 0,1,2$. Thus it holds
\begin{align*}
{\mathrm{IBr}}(B) 
&= 
\{ 
k\mathfrak M\otimes_{k\mathfrak A_4}k_{\mathfrak A_4}, \
k\mathfrak M\otimes_{k\mathfrak A_4}k_{\omega}, \
k\mathfrak M\otimes_{k\mathfrak A_4}k_{\bar\omega}
\}
\\
&=
\{ f^{-1}(W), \ f^{-1}(W_1), \ f^{-1}(W_2)    \}
\end{align*}
as sets. 
Since all simple $k\mathfrak A_4$-modules are
liftable as we know by looking at the $2$-decomposition matrix of 
$\mathfrak A_4$, it holds that all the three simple $kG$-modules
in $B$ are trivial source modules with vertex $P$,
and they uniquely lift from
$k$ to $\mathcal O$, and these afford irreducible ordinary
characters, as in Case~1. 

{\bf Case 3:} $B' = B_0(\mathcal O\mathfrak A_5)$. 
Hence $G' = \mathfrak A_5$ in this case.
Recall that $N_{\mathfrak A_5}(P) = \mathfrak A_4$
by looking at the $2$-decomposition matrices. 
Since $3 = \ell (B) = \ell (b)$, and since $P \trianglelefteq N$,
$b$ is Puig equivalent to $\mathcal O\mathfrak A_4$.
So, we can set 
${\mathrm{IBr}}(b) := \{ W =: W_0, W_1, W_2 \}$.
Obviously, $W_i$ for $i = 0,1,2$ are the all indecomposable 
trivial source $kN$-modules in $b$ with vertex $P$.
Thus, similarly to the previous argument, all the indecomposable trivial source
$kG$-modules in $B$ with vertex $P$ are
$f^{-1}(W_i)$ for $i = 0,1,2$.

On the other hand, all the indecomposable trivial source 
modules in $B_0(\mathcal O\mathfrak A_5)$ with vertex $P$ are
$k_{\mathfrak A_5}$, $5a$ and $5b$, where $5a$ and $5b$ are
the uniserial $k\mathfrak A_5$-modules in Lemma \ref{GreenCorrA5}, 
and we know also that $k_{\mathfrak A_5}\leftrightarrow \chi_1$, 
$5a\leftrightarrow\chi_5$ and $5b\leftrightarrow\chi_5$ 
(with notation as in Lemma \ref{GreenCorrA5}).
Hence, all the indecomposable trivial source $kG$-modules in $B$
with vertex $P$ are
$$
\{ f^{-1}(W), \ f^{-1}(W_1), \ f^{-1}(W_2) \}
=
\{ 
k\mathfrak N\otimes_{k\mathfrak A_5}k_{\mathfrak A_5}, \
k\mathfrak N\otimes_{k\mathfrak A_5}5a, \
k\mathfrak N\otimes_{k\mathfrak A_5}5b
\}
$$
as sets. Therefore, by making use of 
\cite[Exercise (9.7)]{Thevenaz} again, as above,
all of these three indecomposable $kG$-modules uniquely lift
from $k$ to $\mathcal O$ and the afforded ordinary characters are
all irreducible.

(b) From the proof of (a), $B$ is Puig equivalent to
a block algebra $B'$ which is one of 
$\mathcal OP$, $\mathcal O \mathfrak A_4$ or
$B_0(\mathcal O\mathfrak A_5)$. 
Let $\chi'\in{\mathrm{Irr}}(B')$ and $\phi_u'\in{\mathrm{IBr}}(B_u')$
be the characters corresponding to $\chi$ and $\phi_u$, respectively,
via the Puig equivalence.
We know that on the side of $B'$, the generalized $2$-decomposition
number is one, namely, $d^{u}_{\chi',\phi_u'} = 1$.
We know also that any Puig equivalence preserves generalized
decomposition numbers, see 
\cite[Proposition (43.10)]{Thevenaz}. 
Therefore, we finally get the assertion. 
\end{proof}

\section{Proofs of Theorem~\ref{C2xC2new}, Corollary~\ref{etSuffCond} 
and Theorem~\ref{structureT(G)}}\label{ssec:proofs}

We finally turn to the proofs of the main results.

\begin{proof}[Proof of Theorem~\ref{C2xC2new}] 
(a) This follows from Proposition~\ref{IrrChar}(a). 

(b) Let $V$ and $\chi$ be as in Proposition~\ref{IrrChar}.
Take any element $u\in P$ with $u\not= 1$, and let
$B_u^{(1)}, \cdots, B_u^{(n)}$ be all the blocks of $C_G(u)$ 
with $(B_u^{(i)})^G = B$ for $1\leq i\leq n$.
Then, these blocks are nilpotent
since $C_G(u)$ is $2$-nilpotent, so let $\phi_u^{(i)}$ be
the unique irreducible Brauer character in $B_u^{(i)}$
for each $i$. Then, by Proposition~\ref{IrrChar}(b),
$d^u_{\chi, \phi_u^{(i)}} \ = \ 1$ for any $i$.
Hence, by Brauer's 2nd Main Theorem
$$
\chi(u) \ = \ \sum_{i=1}^n d^u_{\chi,\phi_u^{(i)}}{\cdot}\phi_u^{(i)}(1) 
\ = \ \sum_{i=1}^n \phi_u^{(i)}(1).
$$
This shows that $\chi (u)=1$ if and only if $n=1$ and
$\phi_u^{(1)}(1)=1$.
Now, notice that the assumption and Lemma \ref{princ_type} imply that
there is exactly one block of $C_G(P)$ whose block induction to $G$
is $B$.
Thus, since $P\cong C_2\times C_2$, 
$B$ is of principal type if and only if 
$n=1$ for any $u \in P$ with $u\not= 1$
(note that $n$ is determined by $u$).

Assume first that $V$ is endo-trivial. 
Note that $V$ is an indecomposable trivial source $kG$-module
and $p{\not|}\, \dim_k(V)$ by Lemma~\ref{lem:basicET}(h).
Hence, by Theorem~\ref{prop:torchar}, 
$\chi(u)=1$ for any $u\in P$ with $u\not= 1$.
Therefore the above argument implies that $B$ is of principal type
and $\phi_u^{(1)}(1)=1$ for any $u\in P$ with $u\not= 1$.

Conversely, suppose that the two conditions hold.
Then, again by the above argument, 
$\chi(u)=\phi_u(1)=1$ for any $u\in P$ with $u\not= 1$.
Therefore, by Theorem~\ref{prop:torchar}, $V$ is endo-trivial.

(c) Let $V$ be as in the proof of (b). By (b),
the fact that $V$ is endo-trivial depends only on the block
to which $V$ belongs. Hence the assertion holds.  
\end{proof}

\begin{proof}[Proof of Corollary \ref{etSuffCond}]
(a) Set $V:=f^{-1}(1a)$, and let $B$ be a block of $G$ to which
$V$ belongs. By the assumption there is a one-dimensional
$kG$-module $V'$ in $B$. Clearly $V'$ is endo-trivial
by Lemma \ref{lem:basicET}(c). 
Since $f(V') = {\mathrm{Res}}^G_N(V')$ and this is of
dimension one, set $1b:=f(V')$. Obviously, $V'=f^{-1}(1b)$.
Since Theorem \ref{C2xC2new}(b) holds for $f^{-1}(1b)$,
Theorem \ref{C2xC2new}(c) yields that $V$ is endo-trivial.
The latter part is easy.

(b) Set $V := f^{-1}(1a)$, and let $B$ be a block of $G$ to which
$V$ belongs. Let $B_u$ be the same as in Theorem \ref{C2xC2new}.
Then, $B$ has full defect, and hence so dose $B_u$.
Since $C_G(u)$ is $2$-nilpotent, it follows from Morita's theorem
\cite[Lemma 2]{Koshitani1982} that
$$
kB_u \cong {\mathrm{Mat}}_{\theta(1)}(kP)
\qquad \text{as }k{\text{-algebras}}
$$
for some $\theta\in{\mathrm{Irr}}(O_{2'}(C_G(u)))$.
Hence, the assumption implies that $\theta(1) = 1$,
so that $kB_u \cong kP$ as $k$-algebras, which means
$\phi_u(1) = 1$. Thus, Theorem \ref{C2xC2new}(b) yields the
assertion.
\end{proof}

We finally turn to the structure of the group $T(G)$ of 
endo-trivial modules and prove Theorem~\ref{structureT(G)}.
We start by showing that when 
$|N_{G}(P):C_{G}(P)|=3$ the torsion subgroup 
$TT(G)$ is neither trivial, nor equal to $X(G)$ .

\begin{proof}[Proof of Theorem~\ref{structureT(G)}] 
As before set $N:=N_{G}(P)$ amd $\bar{N}:=N/O_{2'}(N)$.
(a) We have $\bar{N}\cong \fA_{4}$ by assumption that 
$|N:C_{G}(P)|=3$,  and by Lemma~\ref{lem:TGKlein}, 
$$
T(\fA_{4})=X(\fA_{4})\oplus 
\left<[\Omega(k)]\right>\cong \IZ/3\IZ\oplus \IZ\,.
$$
Now both the inflation map 
$\Inf_{\bar{N}}^{N}:T(\bar{N})\lra T(N): 
[M]\mapsto [\Inf_{\bar{N}}^{N}(M)]$  
and  the restriction map $\Res^{G}_{N}: T(G)\lra T(N)$ are injective 
group homomorphisms (see Lemma~\ref{lem:basicET}(e)). 
Moreover, as 
$X(\fA_{4})=\{k_{\fA_{4}},k_{\omega},k_{\bar{\omega}}\}$ 
consists of modules all lying in the principal $2$-block of 
$\fA_{4}$, the three modules in  
$f^{-1}(\Inf_{\bar{N}}^{N}(X(\fA_{4})))\cong\IZ/3\IZ$ 
lie in the principal $2$-block of $G$ and they are 
endo-trivial by Corollary~\ref{etSuffCond}.
Thus it follows from Lemma~\ref{lem:TGKlein}(b) 
that \smallskip $\IZ/3\IZ\oplus\IZ\leq T(G)$.

(b) Since $P$ is self-centralizing  and $|N:C_{G}(P)|=3$, 
we have $N_{G}(P)\cong \fA_{4}$. Thus the 
claim is a direct consequence of (a).

(c) This is a direct consequence of Corollary~\ref{etSuffCond}.

(d) By Lemma~\ref{lem:TGKlein}(b), 
we have $T(G)=TT(G)\oplus \IZ$, where $TT(G)=f^{-1}(X(N))\cap T(G)$.
  Clearly $X(G)\cong(G^{ab})_{2'}=C_{n}\times H^{ab}$ 
is a subgroup of $TT(G)$ by Lemma~\ref{lem:basicET}(c).
  Moreover $N\cong (\fA_{4}\rtimes C_{n})\times H$, so that 
$X(N)\cong (N^{ab})_{2'}\cong 
(\fA_{4}\rtimes C_{n})^{ab}\times H^{ab}
\cong (\fA_{4})^{ab}\times C_{n}\times H^{ab}$. 
But by (a) and its proof, 
$f^{-1}((\fA_{4})^{ab})\cong\IZ/3\IZ$ 
is a subgroup of $TT(G)$. This shows that  
$TT(G)=f^{-1}(X(N))\cong \IZ/3\IZ\oplus\IZ/n\IZ\oplus H^{ab}$.
\end{proof}

\begin{rem}
An alternative proof for Corollary~\ref{etSuffCond}(a) 
can be obtained via the position of endo-trivial modules 
in the stable Auslander-Reiten quiver $\Gamma_{s}(kG)$ 
of $kG$ as follows.\par
Let us first look at  the principal block of $G$.   
By Webb's Theorem \cite[Theorem D]{Webb1982}, 
the component $\Theta(k_{G})$ of the trivial module 
in $\Gamma_{s}(kG)$ is isomorphic to the component 
$\Theta(k_{\fA_{4}})$ of the trivial $k\fA_{4}$-module 
in $\Gamma_{s}(k\fA_{4})$ via inflation from $\bar{N}$ 
to $N$ followed by the Green correspondence $f$
with respect to $(G,P,N)$. 
But we see from 
a theorem of Webb and Okuyama 
(\cite[\S4.17]{Benson1998}, \cite{Webb1982}, \cite{Okuyama1987}) 
that the three one-dimensional $k\fA_{4}$-modules 
$k, k_{\omega}, k_{\overline{\omega}}$ all belong to
$\Theta(k_{\fA_{4}})\cong \IZ \widetilde{A}_{5}$.
Furthermore by a result of Bessenrodt 
\cite[Theorem 2.6]{Bessenrodt1991a}, 
all modules in $\Theta(k_{G})$ are 
endo-trivial because $k_{G}$ is. 
In consequence the Green correspondents 
$f^{-1}(\Inf_{\bar{N}}^{N}(k_{\omega}))$ and 
$f^{-1}(\Inf_{\bar{N}}^{N}(k_{\omega}))$ 
are endo-trivial as well, as required.

Similarly if $kB$ is a non-principal block 
of $kG$ containing a one-dimensional $kG$-module 
$1\mathfrak b$, which is endo-trivial by Lemma~\ref{lem:basicET}(c), 
then again by \cite[Theorem 2.6]{Bessenrodt1991a}, 
all modules in the component $\Theta(1\mathfrak b)$ of 
$1\mathfrak b$ in $\Gamma_{s}(kG)$ are endo-trivial.
 Moreover  \cite[Theorem 2.6]{Bessenrodt1991a} 
also says that $\Theta(1\mathfrak b)\cong \IZ \widetilde{A}_{5}$. 
But from the structure of a 
$\IZ \widetilde{A}_{5}$-component there exists two 
non-isomorphic indecomposable  $kB$-modules 
$M\ncong 1\mathfrak b \ncong N$ such that any module lying 
in $\Theta(1\mathfrak b)$ is of the form $\Omega^{r}(1\mathfrak b)$, 
or $\Omega^{m}(M)$, or 
$\Omega^{n}(N)$ for some $r$, $m$,$n\in\IZ$.
 This and Lemma~\ref{lem:TGKlein}(b) show that 
$kB$ contains exactly three torsion endo-trivial, 
that is $kG$-modules of the form 
$1\mathfrak b =f^{-1}(1a),f^{-1}(1{a'}), f^{-1}(1{a''})$ 
for three non-isomorphic one-dimensional $kN$-modules
$1a$, $1{a'}$, $1{a''}$ in the Brauer correspondent 
$b$ of the block $B$.
\end{rem}

\noindent
{\bf Acknowledgements.}
{\small
For this research the first author was partially
supported by the Japan Society for Promotion of Science (JSPS),
Grant-in-Aid for Scientific Research (C)23540007, 2011--2014. 
The second author gratefully acknowledges financial support by ERC
Advanced Grant 291512 and SNF Fellowship 
for Prospective Researchers 
PBELP2$_{-}$143516.
This research started when the second author was visiting Chiba
University March 2014, which 
was supported by the Japan Society for Promotion of Science (JSPS),
Grant-in-Aid for JSPS-Fellows 24.2274, via Dr.~Moeko Takahashi, Chiba
University, whom both authors are grateful to.
The authors are sincerely grateful to Erwan Biland for enlightening discussions.
}

\end{document}